\documentclass[a4paper,11pt]{article}
\usepackage[utf8]{inputenc}  
\usepackage[T1]{fontenc}
\usepackage[french,english]{babel}
\usepackage{ifthen}
\usepackage{array}
\usepackage{layout}
\usepackage{tabularx}
\usepackage{supertabular}
\usepackage{setspace}
\usepackage{lmodern}
\usepackage{graphicx}
\usepackage{epstopdf}
\usepackage{epsfig}
\usepackage{enumerate}
\usepackage{amsfonts}
\usepackage{amsmath}
\usepackage{amssymb}
\usepackage{amsthm}
\usepackage{mathrsfs}
\usepackage{algorithm,algorithmic}
\usepackage{tikz}
\usepackage{float} 
\usepackage[toc,page]{appendix}
\usepackage{framed}
\usepackage{calc}
\usepackage{caption}
\usepackage{cases}
\usepackage{multicol}

\newtheorem{Thm}{Theorem}[section]
\newtheorem{Lemme}[Thm]{Lemma}
\newtheorem{Prop}[Thm]{Proposition}

\newtheorem{remark}[Thm]{Remark}

\newcommand\R{\mathbb{R}}

\newcommand\T{\mathbb{T}}

\newcommand\Z{\mathbb{Z}}

\newcommand\Ng[1]{\left| #1 \right|_{\gamma}}
\newcommand\Ngs[2]{\left| #1 \right|_{\gamma^{#2}}}
\newcommand\Hg{H_{\gamma}(\R)}
\newcommand\Hgs[1]{H_{\gamma^{#1}}(\R)}
\newcommand\Nb[1]{\left| #1 \right|_{\beta}}
\newcommand\Hb{H_{\beta}(\R)}
\newcommand\abs[1]{\left|#1\right|}
\newcommand\Normp[2]{\left\lVert#1\right\rVert_{L^#2}}

\newenvironment{changemargin}[2]{\begin{list}{}{%
\setlength{\topsep}{0pt}%
\setlength{\leftmargin}{0pt}%
\setlength{\rightmargin}{0pt}%
\setlength{\listparindent}{\parindent}%
\setlength{\itemindent}{\parindent}%
\setlength{\parsep}{0pt plus 1pt}%
\addtolength{\leftmargin}{#1}%
\addtolength{\rightmargin}{#2}%
}\item }{\end{list}}

\usepackage{geometry}
\title{KdV}
\author{Pierre GARNIER}
\date{2013}

\begin{document}

\title{\Large {\bf Damping to prevent the blow-up of the Korteweg-de Vries equation} 
} 
 
\author{ {Pierre Garnier\footnote{pierre.garnier@u-picardie.fr}}\\ 
{\small Laboratoire Ami\'enois de Math\'ematique Fondamentale et Appliqu\'ee,} \\  {\small CNRS UMR 7352, Universit\'e 
de Picardie Jules Verne,}\\
{\small 80069 Amiens, France.}
}

\date{ }
\maketitle

\begin{minipage}[t]{14cm}
	{\footnotesize {\bf Abstract.} We study the behavior of the solution of a generalized damped KdV equation $u_t + u_x + u_{xxx} + u^p u_x + \mathscr{L}_{\gamma}(u)= 0$. We first state results on the local well-posedness. Then when $p \geq 4$, conditions on $\mathscr{L}_{\gamma}$ are given to prevent the blow-up of the solution. Finally, we numerically build such sequences of damping.} 
\end{minipage} \\ 
\\      
     
\begin{minipage}[t]{10.5cm}{\footnotesize {\bf Keywords.} 
KdV equation, dispersion, dissipation, blow-up.}
\end{minipage}\\

{\footnotesize {\bf MS Codes.} 35B44, 35Q53, 76B03, 76B15.}

\section*{Introduction}

	The Korteweg-de Vries (KdV) equation is a model of one-way propagation of small amplitude, long wave \cite{KdV}. It is written as
\[ u_t + u_x + u_{xxx} + u u_x = 0. \]

In \cite{BonDouKar}, Bona et al. consider the initial- and periodic-boundary-value problem for the generalized Korteweg-de-Vries equation 
\[ u_t + u_{xxx} + u^p u_x = 0 \]
and study the effect of a dissipative term on the global well-posedness of the solution. Actually, they consider two different dissipative terms, a Burgers-type one $-\delta u_{xx}$ and a zeroth-order term $\sigma u$. For both these terms, they show that for $p \geq 4$, there exist critical values $\delta_c$ and $\sigma_c$ such that if $\delta>\delta_c$ or $\sigma>\sigma_c$ the solution is globally well-defined. However, the solution blows-up when the damping is too weak as for the KdV equation \cite{MarMer}.The literature is full of work concerning the dampen KdV equation with $p=1$ \cite{AmBonaSch,CabRosa,CheSad,CheSad2,Ghi,Ghi2,Goub,GoubRosa}, but few are concerning more general nonlinearities.

In our paper, we consider a more general damping term denoted by $\mathscr{L}_{\gamma}(u)$. Our purpose is to find similar results as above, both theorically and numerically. So the KdV equation becomes a damped KdV (dKdV) equation and is written
\[ u_t + u_x + u_{xxx} + u^p u_x + \mathscr{L}_{\gamma}(u)= 0. \]
The damping operator $\mathscr{L}_{\gamma}(u)$ works on the frequencies. It is defined by its Fourier symbol
\[ \widehat{\mathscr{L}_{\gamma}(u)}(\xi) :=  \gamma(\xi) \hat{u}(\xi) .\]
Here $\hat{u}$ is the Fourier transform of $u$ and $\gamma$ a strictly positive function chosen such that  
\[ \int_\R  u(x) \mathscr{L}_\gamma (u) d\mu(x) = \int_{\R} \gamma(\xi) |\hat{u}(\xi)|^2 d\xi \geq 0. \]
We notice than the two cases studied in \cite{BonDouKar} are present with this damping by taking $\gamma(\xi) = \delta \xi^2$ and $\gamma(\xi) = \sigma$ respectively.

The KdV equation has an infinite number of invariants such that the $L^2$-norm. But, for the dKdV equation, the $L^2$-norm decreases. Indeed, for all $t \in \R$, 
\[ \frac{d}{dt} \Normp{u}{2}^2 = - \Ng{u}^2 \]
where the natural space of study is 
\[ \Hg := \left\{ u \in L^2(\R) \mbox{~s.t.~} \int_{\R} \gamma(\xi) \abs{\hat{u}(\xi)}^2 d\xi < +\infty \right\} \]
and the associated norm is
\[ \Ng{u} := \sqrt{\int_{\R} \gamma(\xi) \abs{\hat{u}(\xi)}^2 d\xi}. \]

An other property of the KdV equation is that the solution can blow-up as soon as $p \geq 4$ . The blow-up is caracterized by $\lim\limits_{t \rightarrow T} \left\lVert u \right\rVert_{H^1} = +\infty$.

In this paper, we first establish the local well-posedness of the dKdV equation. Then we study the global well-posedness. More precisely, we focus on the behavior of the $H^1$-norm with respect to $p$ and we obtain conditions on $\gamma$ so there is no blow-up. Finally, we illustrate the results using some numerical simulations. We first find a constant damping ($\gamma(\xi)=$constant) such that there is no blow-up and then the damping is weaken in such a way $\lim\limits_{|\xi| \rightarrow + \infty} \gamma(\xi) = 0$.

\section{Preliminary results}

	Some results of injection concerning the space $\Hg$ are given.

\begin{Prop}\label{prop1.1}
	Assume $\int_{\R} \frac{1}{\gamma(\xi)} < +\infty$ then there exists a constant $C>0$ such that $\| u \|_{\infty} \leq C \Ng{u}$, i.e., the injection $\Hg \hookrightarrow L^{\infty}(\R)$ is continuous.
\end{Prop}	

\begin{proof}
	Let $u \in H_{\gamma}(\R)$. We notice that 
	\[ u(x)=\int_{\R} \hat{u}(\xi) e^{i\xi x} d\xi. \]
	Then
	\[ |u(x)| \leq \int_{\R} |\hat{u}(\xi)| =  \int_{\R} \frac{1}{\sqrt{\gamma(\xi)}} \sqrt{\gamma(\xi)} |\hat{u}(\xi)| .\]
	We assumed that $\gamma(\xi) > 0$. Hence, the Cauchy-Schwarz inequality involves for all $x \in \R$ :
	\[ |u(x)| \leq \left( \int_{\R} \frac{1}{\gamma(\xi)} \right)^{\frac{1}{2}} \left( \int_{\R} \gamma(\xi) |\hat{u}(\xi)|^2 \right)^{\frac{1}{2}} = \left( \int_{\R} \frac{1}{\gamma(\xi)} \right)^{\frac{1}{2}} \Ng{u}. \]
\end{proof}


\begin{Prop}
	Let $\gamma$ and $\beta$ be such that for all $\xi \in \R$, $\gamma(\xi)>\beta(\xi)$. We define 
	\[ \rho(N) := \max\limits_{\xi \geq N} \frac{\beta(\xi)}{\gamma(\xi)}. \]
	The continuous injection $\Hg \hookrightarrow \Hb$ is compact if and only if $\lim\limits_{N \rightarrow +\infty}\rho(N)=0$.
\end{Prop}

\begin{proof}
	The condition is necessary. Indeed, if there exists $\alpha>0$ such that $\rho(N) > \alpha, \ \forall N$, then the norms $\Nb{u}$ and $\Ng{u}$ are equivalent, the injection cannot be compact.
	Let us prove now that the condition is sufficient.	
	First, we have for $u \in \Hg$ :
	\[ \Nb{u} = \int_{\R} \beta(\xi) |\hat{u}(\xi)|^2 \leq \int_{\R} \gamma(\xi) |\hat{u}(\xi)|^2 = \Ng{u}.\]
	This shows that the injection is continuous. Now we prove that the injection is compact. We use finite rank operators and we take the limit. Let $I_N$ be the orthogonal operator on the polynomials of frequencies $\xi$ such that $-N \leq \xi \leq N$. We have
	\[ I_N u = \int_{|\xi| \leq N} \hat{u}(\xi) e^{i \xi x} d\xi. \]
	Thus
	\begin{align*}
		\Nb{(Id-I_N)u}^2 & = \int_{|\xi| > N} \beta(\xi) |\hat{u}(\xi)|^2, \\
			 & \leq \int_{|\xi| > N} \frac{\beta(\xi)}{\gamma(\xi)} \gamma(\xi) |\hat{u}(\xi)|^2, \\
			 & \leq \rho(N) \Ng{u}^2 \underset{N \rightarrow + \infty}{\longrightarrow} 0.
	\end{align*}
	Therefore $Id$ is a compact operator and consequently the injection is compact.		 
\end{proof}

\begin{Prop}
	Assume that $u, \ v \in \Hg$ and there exists a constant $C>0$ such that $\forall \xi, \ \eta \in \R$ we have 
	\[ \sqrt{\gamma(\xi)} \leq C \left( \sqrt{\gamma(\xi - \eta)} + \sqrt{\gamma(\eta)} \right). \]
	Then we have
	\[ \Ng{uv} \leq C \left( \Ng{u} \| \hat{v} \|_{L^1} + \Ng{v} \| \hat{u} \|_{L^1}  \right). \]
	Moreover if $\int_{\R} \frac{1}{\gamma(\xi)} < +\infty$ then $\Hg$ is an algebra.
\end{Prop}

\begin{proof}
	Let $u, \, v \in H_{\gamma}(\R)$. We have
	\[ |uv|_{\gamma}^2 = \int_{\R} \gamma(\xi) |\widehat{uv}(\xi)|^2. \]
	We remind that $\widehat{uv}(\xi) = \hat{u} \ast \hat{v} (\xi)$. Using the inequality 
	\[ \sqrt{\gamma(\xi)} \leq C \left( \sqrt{\gamma(\xi - \eta)} + \sqrt{\gamma(\eta)} \right),\]
	we obtain for all $\xi, \ \eta \in \R$
	\[ \sqrt{\gamma(\xi)} |\widehat{uv}(\xi)| \leq C \left( \int_{\R} \sqrt{\gamma(\xi-\eta)} |\hat{u}(\xi-\eta) \hat{v}(\eta)| d\eta + \int_{\R} \sqrt{\gamma(\eta)} |\hat{u}(\xi-\eta) \hat{v}(\eta)| d\eta \right). \]
	Hence
	\begin{align*}
		|uv|_{\gamma}^2 & \leq C^2 \int_{\R} \left( \int_{\R} \sqrt{\gamma(\xi-\eta)} |\hat{u}(\xi-\eta) \hat{v}(\eta)| d\eta + \int_{\R} \sqrt{\gamma(\eta)} |\hat{u}(\xi-\eta) \hat{v}(\eta)| d\eta \right)^2 d\xi, \\
			& \leq C^2 \int_{\R} \left[ \left( \int_{\R} \sqrt{\gamma(\xi-\eta)} |\hat{u}(\xi-\eta) \hat{v}(\eta)| d\eta \right)^2 + \left( \int_{\R} \sqrt{\gamma(\eta)} |\hat{u}(\xi-\eta) \hat{v}(\eta)| d\eta \right)^2 \right] d\xi, \\
			& \leq C^2 \left( \left\lVert \left( \sqrt{\gamma(\xi)}|\hat{u}| \right) \ast |\hat{v}| \right\rVert_{L^2}^2 + \left\lVert |\hat{u}| \ast \left( \sqrt{\gamma(\xi)} |\hat{v}| \right) \right\rVert_{L^2}^2  \right).
	\end{align*}
	However, for $f \in L^1$ and $g \in L^2$, we have
	\[  \| |f| \ast |g| \|_{L^2}^2 \leq \| g \|_{L^2}^2 \| f \|_{L^1}^2. \]		
	Thus
	\[ |uv|_{\gamma}^2 \leq C \left( |u|_{\gamma}^2 \|\hat{v}\|_{L^1}^2 + |v|_{\gamma}^2 \|\hat{u}\|_{L^1}^2 \right). \]
	From proposition \ref{prop1.1}, we know there exists a constant $c>0$ such that $\| \hat{u} \|_{L^1} \leq c |u|_{\gamma}$ if $\int_{\R} \frac{1}{\gamma(\xi)} < +\infty$. Then, there exists $\tilde{C}>0$ such that
	\[ |uv|_{\gamma} \leq \tilde{C} |u|_{\gamma} |v|_{\gamma}.\]
\end{proof}

\section{Local well-posedness}

	We study the following Cauchy problem : $\forall x \in \R, \ \forall t>0$,
\begin{numcases}{}
	u_t + u_x + u_{xxx} + u^p u_x + \mathscr{L}_{\gamma} (u) = 0, \label{pb_Cauchy1} \\
	u(x,t=0) = u_0(x). \label{pb_Cauchy2}
\end{numcases}	
The semi-group generated by the linear part is written as	
\[ S_t u := \int_{\R} e^{i\xi x} e^{i(\xi^3-\xi)t-\gamma(\xi)t} \hat{u}(\xi) d\xi. \]
In the rest of the section, $f(u)$ denotes the non-linear part of the equation, i.e., $f(u) = u^p u_x$.
We first state a result of regularization.
\begin{Lemme}\label{lemme_ex_loc}
	Assume that $s , \ r \in \R^+$. Then there exists a constant $C_r > 0$, depending only on $r$, such that $\forall u \in \Hgs{s}$ and $\forall t>0$ we have
	\[ \Ngs{S_t u}{s+r}^2 \leq \frac{C_r}{t^r} \Ngs{u}{s}^2 . \]
\end{Lemme}

\begin{proof}
	Let $r \in \R^+$, $u \in \Hgs{s}$ and $t>0$. Then we have
	\begin{align*}
		\Ngs{S_t u}{s+r}^2 & = \int_{\R} \gamma(\xi)^{s+r} \left| e^{-\gamma(\xi)t} \hat{u}(\xi) \right|^2 d\xi \\
			& \leq \sup\limits_{\xi \in \R} \left( \gamma(\xi)^r e^{-2\gamma(\xi)t} \right) \Ngs{u}{s}^2.
	\end{align*}
	But $\forall \xi \in \R$
	\[ \gamma(\xi)^r e^{-2\gamma(\xi)t} \leq \frac{\left( \frac{r}{2} \right)^r e^{-r}}{t^r} = \frac{C_r}{t^r}. \]
	Thus
	\[ \Ngs{S_t u}{s+r}^2 \leq \frac{C_r}{t^r} \Ngs{u}{s}^2. \]
\end{proof}

\begin{Thm}\label{thm_local}
	 Assume that there exists $r\in ]0,2[$ and for all $\xi \in \R$, $\gamma(\xi) \geq \xi^{\frac{2}{r}}$. We also assume that $\int_{\R} \frac{1}{\gamma(\xi)^s} < +\infty$ and there exists a constant $C>0$ such that $\forall \xi, \ \eta \in \R$ and $s \in \R^+$ we have
	\[ \sqrt{\gamma(\xi)^s} \leq C \left( \sqrt{\gamma(\xi - \eta)^s} + \sqrt{\gamma(\eta)^s} \right). \]
	Then there exists a unique solution in $\mathscr{C}\left( [-T,T],\Hgs{s} \right)$ of the Cauchy problem \eqref{pb_Cauchy1}-\eqref{pb_Cauchy2}.
	
	Moreover, for all $M>0$ with $\Ngs{u_0}{s} \leq M$ and $\Ngs{v_0}{s} \leq M$, there exists a constant $C_1>0$ such that the solution $u$ and $v$, associated with the initial data $u_0$ and $v_0$ respectively, satisfy for all $t \leq \left( \frac{1}{C_0 M^p} \right)^{\frac{2}{r}}$
	\[ \Ngs{u(\cdot,t) - v(\cdot,t)}{s} \leq C_1 \Ngs{u_0 - v_0}{s}. \] 
\end{Thm}

\begin{proof}
	Thanks to Duhamel's formula, $\Phi(u)$ is solution of the Cauchy problem, where
	\[ \Phi(u) = S_t u_0 - \int_0^t S_{t-\tau} f(u(\tau)) d\tau . \]
	Let show that $u$ is the unique fixed-point of $\Phi$. 
	We introduce the closed ball $\bar{B}(T)$ defined for $T>0$ by 
	\[ \bar{B}(T) := \left\{ u \in \mathscr{C}\left( [0,T];\Hgs{s} \right) \mbox{~s.t.~} \Ngs{u(t)-u_0(t)}{s} \leq 3 \Ngs{u_0}{s} \right\}. \]
	We apply the Picard fixed-point theorem. We first show that $\Phi \left( \bar{B}(T) \right) \subset \bar{B}(T)$. Let us take $u \in \bar{B}(T)$ and show that $\Phi (u(t)) \in \bar{B}(T)$. We have
	\[ \Ngs{\Phi(u(t))}{s} \leq \Ngs{S_t u_0}{s} + \int_0^t \Ngs{S_{t-\tau}f(u(\tau))}{s}. \]
	On the one hand, we have
	\[ \Ngs{S_t u_0}{s}^2 = \int_{\R} \gamma(\xi)^s \left| \widehat{S_t u_0} \right|^2 \leq \int_{\R} \gamma(\xi)^s \left| \hat{u_0} \right|^2 \leq \Ngs{u_0}{s}^2. \]
	On the other hand, we apply Lemma \ref{lemme_ex_loc}
	\begin{align*}
		\Ngs{S_{t-\tau} f(u(\tau))}{s} & = \Ngs{S_{t-\tau} f(u(\tau))}{s-r+r} \\
			& \leq \frac{C_r}{(t-\tau)^{\frac{r}{2}}} \Ngs{f(u(\tau))}{s-r}.
	\end{align*}
	But
	\begin{align*}
		\Ngs{f(u(\tau))}{s-r}^2 & = \frac{1}{(p+1)^2} \int_{\R} \frac{\xi^2}{\gamma(\xi)^r} \gamma(\xi)^s \left| \widehat{u^{p+1}} \right|^2 d\xi \\
			& \leq \frac{1}{(p+1)^2} \Ngs{u^{p+1}}{s}^2 \\
			\intertext{because $\gamma(\xi)>\xi^{\frac{2}{r}}$, and $\Hgs{s}$ beeing an algebra, we have}
		\Ngs{f(u(\tau))}{s-r} & \leq C \Ngs{u}{s}^{p+1}.
	\end{align*}
	Consequently
	\begin{align*}
		\Ngs{\Phi(u)}{s} & \leq \Ngs{u_0}{s} + \int_0^t \frac{C}{(t-\tau)^{\frac{r}{2}}} \Ngs{u}{s}^{p+1} d\tau \\
			& \leq \Ngs{u_0}{s} + C \sup\limits_{t \in [0,T]} \left( \Ngs{u}{s}^{p+1} \right) \int_0^t \frac{1}{(t-\tau)^{\frac{r}{2}}} d\tau \\
			& \leq \Ngs{u_0}{s} + \frac{C}{1-\frac{r}{2}} T^{1-\frac{r}{2}} \sup\limits_{t \in [0,T]} \left( \Ngs{u}{s}^{p+1} \right).
	\end{align*}
	But $u \in \bar{B}(T)$, then we have
	\[ \Ngs{u(t)}{s} - \Ngs{u_0}{s} \leq \Ngs{u(t)-u_0}{s} \leq 3 \Ngs{u_0}{s}. \]
	That involves
	\[ \Ngs{u(t)}{s} \leq 4 \Ngs{u_0}{s}, \]
	and
	\[ \sup\limits_{t \in [0,T]} \left( \Ngs{u}{s}^{p+1} \right) \leq \left( \sup\limits_{t \in [0,T]}  \Ngs{u}{s} \right)^{p+1} \leq 4^{p+1} \Ngs{u_0}{s}^{p+1}. \]
	We have $\Phi(u(t)) \in \bar{B}(T)$ if the inequality 
	\[ \Ngs{\Phi(u(t)) - u_0}{s} \leq 2 \Ngs{u_0}{s} + C_r T^{1-\frac{r}{2}} \left( 4^{p+1} \Ngs{u_0}{s}^{p+1} \right) \leq 3 \Ngs{u_0}{s}  \]
	is true i.e. if
	\[ 0 < T^{1-\frac{r}{2}} \leq \frac{1}{4^{p+1} C_r \Ngs{u_0}{s}^p}. \]

	Now let us show that $\Phi$ is a strictly contracting map. Let $u, \ v \in \bar{B}(T)$, we prove that $\forall t \in [0,T]$,
	\[ \sup\limits_{t \in [0,T]} \Ngs{\Phi(u(t)) - \Phi(v(t))}{s} \leq k \sup\limits_{t \in [0,T]} \Ngs{u-v}{s} \]
	with $k \in [0,1[$.
	As previously, we have
	\begin{align*}
		\Ngs{\Phi(u(t)) - \Phi(v(t))}{s} & = \Ngs{\int_0^t S_{t-\tau} \left( f(u(\tau)) - f(v(\tau)) \right) d\tau}{s} \\
			& \leq \int_0^t \frac{C_0}{(t-\tau)^{\frac{r}{2}}} \Ngs{u^{p+1} - v^{p+1}}{s}.
	\end{align*}
	Using the equality 
	\[ u^{p+1} - v^{p+1} = (u-v) \sum\limits_{i+j=p}u^i v^j \]
	and the injection results, we obtain
	\begin{align*}
		\Ngs{u^{p+1} - v^{p+1}}{s} & \leq C_1 \Ngs{u-v}{s} \Ngs{\sum\limits_{i+j=p}u^i v^j}{s} \\
			& \leq C_2 \Ngs{u-v}{s} \sum\limits_{i+j=p} \Ngs{u}{s}^i \Ngs{v}{s}^j \\
			& \leq C_3 \Ngs{u-v}{s} \Ngs{u_0}{s}^p.
	\end{align*}
	Then we have
	\begin{align*}
		\sup\limits_{t \in [0,T]} \Ngs{\Phi(u(t)) - \Phi(v(t))}{s} & \leq C \Ngs{u_0}{s}^p \int_0^t \frac{\Ngs{u-v}{s}}{(t-\tau)^{\frac{r}{2}}} d\tau \\
			& \leq C \Ngs{u_0}{s}^p T^{1-\frac{r}{2}} \sup\limits_{t \in [0,T]} \left( \Ngs{u-v}{s} \right). 
	\end{align*}
	The map $\Phi$ is strictly contracting if
	\[ T^{1-\frac{r}{2}} < \frac{1}{C\Ngs{u_0}{s}^p}. \]
	
	It remains to prove the continuity with respect to the initial data. Duhamel's formula gives for $t \in [0,T]$, $T^{1-\frac{r}{2}} \leq \frac{1}{C_0 M^p}$ 
	\begin{align*}
		\Ngs{u-v}{s} & \leq \Ngs{u_0 - v_0}{s} + \int_0^t \Ngs{f(u)-f(v)}{s} d\tau \\
		& \Ngs{u_0 - v_0}{s} + C' T^{1-\frac{r}{2}} \left( \sum\limits_{i+j=p} \Ngs{u_0}{s}^i \Ngs{v_0}{s}^j \right) \Ngs{u-v}{s} \\
		& \Ngs{u_0 - v_0}{s} + C' T^{1-\frac{r}{2}} \left( \sum\limits_{i+j=p} \Ngs{u_0}{s}^i \Ngs{v_0}{s}^j \right) \sup\limits_{t \in [0,T]} \left( \Ngs{u-v}{s} \right).
	\end{align*}
	It involves
	\[ \Ngs{u-v}{s} \leq C_1 \Ngs{u_0-v_0}{s}. \]	
\end{proof}

\begin{remark}
	Actually we can proove the local well-posedness for every $\gamma$ using a parabolic regularisation
	\[ u_t + u_x + u_{xxx} + u^p u_x + \mathscr{L}_{\gamma} (u) - \epsilon u_{xx} = 0. \]
	Using lemma \ref{lemme_ex_loc} with $\gamma(\xi) = \xi^2$, the same computations as theorem \ref{thm_local} and taking the limit $\epsilon \rightarrow 0$ give the result \cite{Iorio,BonaSmith1}.
\end{remark}

\section{Global well-posedness}
	
	We work here under the hypothesis of the local theorem and study the global well-posedness of the damped KdV equation. We use here an energy method \cite{BonaSmith1,BonaSmith2}

	\begin{Thm}\label{thm_global}
		If $p<4$, for all $\gamma$, the unique solution is global in time, valued in $H^1(\R)$.
		Else ($p\geq 4$), there exists a constant $\theta >0$ such that if $\gamma(\xi) \geq \theta, \ \forall \xi \in \R$ then the unique solution is global in time, valued in $H^2(\R)$.
	\end{Thm}
	
	\begin{proof}	
		\noindent\textbf{Case $p<4$:}
		We begin by introducing $N(u)$ and $E(u)$, two invariants of the KdV equation without the damping term, which are the $L^2$-norm and the energy. Their expressions are 
		\begin{align*}
			& N(u) = \int_{\R} u^2 dx = \Normp{u}{2}^2, \\
			& E(u) = \frac{1}{2}\int_{\R} u_x^2 dx - \frac{1}{(p+1)(p+2)}\int_{\R}u^{p+2}dx = \frac{1}{2} \Normp{u_x}{2}^2 - \frac{1}{(p+1)(p+2)} \Normp{u}{p+2}^{p+2}. 
		\end{align*}		 
		We first multiply \eqref{pb_Cauchy1} by $u$ and we integrate with respect to $x$. Then we have
		\[ \frac{1}{2}\frac{d}{dt}\int_{\R}u^2 dx + \int_{\R}\mathscr{L}_{\gamma}(u)u dx = 0. \]
		Integrating with respect to time, we obtain
		\[ \int_{\R}u^2 dx + 2\int_0^t \left( \int_{\R}\mathscr{L}_{\gamma}(u)u dx \right)d\tau = \int_{\R} u_0^2 dx. \]
		Which can also be written as
		\[ N(u) + 2\int_0^t \Ng{u}^2 d\tau = N(u_0). \]
		We deduce from that expression that $N(u)$ is a decreasing function and $\int_0^t \Ng{u}^2 d\tau$ is bounded independently of $t$ by $N(u_0)$.
		Now, we multiply \eqref{pb_Cauchy1} by $u_{xx} + \frac{u^{p+1}}{p+1}$ and we integrate with respect to $x$. Then we have
		\[ \frac{d}{dt} \left( \int_{\R} -\frac{u_x^2}{2}+\frac{u^{p+2}}{(p+1)(p+2)} dx \right) - \int_{\R}\mathscr{L}_{\gamma}(u_x)u_x dx + \int_{\R}\mathscr{L}_{\gamma}(u) \left( \frac{u^{p+1}}{p+1} \right) dx = 0. \]
		Integrating with respect to time, we obtain
		\[ E(u) + \int_0^t \Ng{u_x}^2 d\tau - \int_0^t \left( \int_{\R} \mathscr{L}_{\gamma}(u) \left( \frac{u^{p+1}}{p+1} \right) dx \right)d\tau = E(u_0). \]
		From this expression, we have
		\begin{align*}
			\int_{\R} u_x^2 dx & = E(u) + \int_{\R}\frac{u^{p+2}}{(p+1)(p+2)} \\
				& \leq E(u_0) + \int_{\R}\frac{u^{p+2}}{(p+1)(p+2)} + \int_0^t \left( \int_{\R}\mathscr{L}_{\gamma}(u) \left( \frac{u^{p+1}}{p+1} \right) dx \right)d\tau \\
				& \leq E(u_0) + \frac{1}{(p+1)(p+2)} \Normp{u}{\infty}^p \Normp{u}{2}^2 + \left( \sup\limits_{0 \leq \tau \leq t} \Normp{u}{\infty}^p \right) \frac{1}{p+1} \int_0^t \left( \int_{\R}\mathscr{L}_{\gamma}(u) u dx \right)d\tau.
		\end{align*}		
		\noindent Using the inequality $\Normp{u}{\infty}^2 \leq 2 \Normp{u}{2} \Normp{u_x}{2}$ and because $\int_{\R}\abs{\mathscr{L}_{\gamma}(u) u} = \int_{\R}\mathscr{L}_{\gamma}(u) u$,
		\[ \int_{\R} u_x^2 dx \leq E(u_0) + \frac{2^{\frac{p}{2}}}{(p+1)(p+2)} \Normp{u}{2}^{2+\frac{p}{2}} \Normp{u_x}{2}^{\frac{p}{2}} + \sup\limits_{0 \leq \tau \leq t} \left( 2^{\frac{p}{2}} \Normp{u}{2}^{\frac{p}{2}} \Normp{u_x}{2}^{\frac{p}{2}} \right) \int_0^t \Ng{u}^2 d\tau. \]
		\noindent Since $\Normp{u}{2} \leq \Normp{u_0}{2}$
		\[ \int_{\R} u_x^2 dx \leq C_0 + C_1 \Normp{u_x}{2}^{\frac{p}{2}} + C_2 \sup\limits_{0 \leq \tau \leq t} \Normp{u_x}{2}^{\frac{p}{2}}. \]
		Then
		\begin{equation}\label{eq_glob0}
			\sup\limits_{0 \leq \tau \leq t} \Normp{u_x}{2}^2 - C \sup\limits_{0 \leq \tau \leq t} \Normp{u_x}{2}^{\frac{p}{2}} \leq C_0. 
		\end{equation}
		If there exists $T>0$ such that $\lim\limits_{t \rightarrow T} \Normp{u_x}{2} = +\infty$ then $\Normp{u_x}{2}^2 - C\Normp{u_x}{2}^{\frac{p}{2}} \rightarrow +\infty$ since $p < 4$ and this is impossible because of \eqref{eq_glob0}. Consequently, $\Normp{u_x}{2}$ is bounded for all $t$ and so is the $H^1$-norm.

		\vspace{0.5cm}
		\noindent\textbf{Case $p \geq 4$:}

		
		We estimate the $L^2$-norm of $u_{xx}$. We multiply \eqref{pb_Cauchy1} with $u_{xxxx}$ and we integrate with respect to $x$. Then we have
		\begin{equation}\label{egal_glob2} 
			\frac{1}{2}\frac{d}{dt} \left( \int_{\R}u_{xx}^2 dx \right) + \int_{\R}\mathscr{L}(u) u_{xxxx} dx = -\int_{\R} u^p u_x u_{xxxx} dx.
		\end{equation}
		Using two integrations by part, we have
		\[ \int_{\R}\mathscr{L}(u) u_{xxxx} dx = \int_{\R}\mathscr{L}(u_{xx}) u_{xx} dx. \]
		Let us work on the last term. Using integrations by part, we have
		\[ -\int_{\R} u^p u_x u_{xxxx} dx = -\frac{5p}{2} \int_{\R} u^{p-1}u_x u_{xx}^2 dx - p(p-1) \int_{\R} u^{p-2} u_x^3 u_{xx} dx. \]
		It follows that
		\[ -\int_{\R} u^p u_x u_{xxxx} dx  \leq \frac{5p}{2} \|u\|_{\infty}^{p-1} \| u_x \|_{\infty} \| u_{xx} \|_{L^2}^2 + p(p-1) \| u \|_{\infty}^{p-2} \| u_x \|_{\infty}^2 \int_{\R} |u_x u_{xx}|dx. \]
		But, from the Cauchy-Schwarz inequality
		\[ \int_{\R} |u_x u_{xx}|dx \leq \| u_x \|_{L^2} \| u_{xx} \|_{L^2}. \]
		Then we have
		\[ -\int_{\R} u^p u_x u_{xxxx} dx  \leq \frac{5p}{2} \|u\|_{\infty}^{p-1} \| u_x \|_{\infty} \| u_{xx} \|_{L^2}^2 + p(p-1) \| u \|_{\infty}^{p-2} \| u_x \|_{\infty}^2 \| u_x \|_{L^2} \| u_{xx} \|_{L^2}. \]
		Using the inequalty $\| u \|_{\infty}^2 \leq 2 \| u \|_{L^2} \| u_x \|_{L^2}$, we obtain
		\begin{align*}
			-\int_{\R} u^p u_x u_{xxxx} dx \leq & \left[ \frac{5p}{2} \left( 2 \| u \|_{L^2}^{\frac{3}{2}} \| u_{xx} \|_{L^2}^{\frac{1}{2}} \right)^{\frac{p-1}{2}} \| u \|_{L^2}^{\frac{1}{4}} \| u_xx \|_{L^2}^{\frac{3}{4}} \right. \\
			& \left. + p(p-1) \left( 2 \| u \|_{L^2}^{\frac{3}{2}} \| u_{xx} \|_{L^2}^{\frac{1}{2}} \right)^{\frac{p-2}{2}} \| u \|_{L^2} \| u_{xx}\|_{L^2} \right] \| u_{xx} \|_{L^2}^{2}. \\
			& =: \Omega \left( \| u \|_{L^2}, \| u_{xx}\|_{L^2} \right) \| u_{xx} \|_{L^2}^{2}.
		\end{align*}
		From \eqref{egal_glob2}, it leads to the inequality
		\[ \frac{1}{2} \frac{d}{dt} \| u_{xx}\|_{L^2} + \int_{\R} \mathscr{L}_{\gamma}(u_{xx})u_{xx} - \Omega u_{xx}^2 dx \leq 0. \]
		But 
		\begin{align*}
			\int_{\R} \mathscr{L}_{\gamma}(u_{xx})u_{xx} - \Omega u_{xx}^2 dx & = \int_{\R} \left[ \widehat{\mathscr{L}_{\gamma}(u_xx)} \overline{\widehat{u_{xx}}} - \Omega \widehat{u_{xx}} \overline{\widehat{u_{xx}}} \right]d\xi \\
			& = \int_{\R} \left( \gamma(\xi) - \Omega \right) \left| \widehat{u_{xx}} \right|^2 d\xi.
		\end{align*}
		The function $\Omega$ is increasing for its two arguments. We previously notice that $\| u(\cdot,t) \|_{L^2}$ is an decreasing function with respect to the time. Then, if $\gamma(\xi) - \Omega \vert_{t=0} \geq 0$, $\| u_{xx}(\cdot,t) \|_{L^2}$ does not increase for $t \geq 0$. Particularly, if $\gamma(\xi) \geq \Omega \left( \| u_0 \|_{L^2}, \| u_{0xx} \|_{L_2} \right) =: \theta$, the semi-norm $\| u_{xx}(\cdot,t) \|_{L^2}$ is bounded by its values at $t=0$.
	\end{proof}

	\begin{remark}
		This result is also true on the torus $\T(0,L)$ where the operator $\mathscr{L}_\gamma$ is defined by its Fourier symbol  
		\[ \widehat{\mathscr{L}_\gamma (u )} (k) := \gamma_k \hat u_k. \]
		Here $\hat u_k$ is the $k-$th Fourier coefficient of $u$ and $(\gamma_k)_{k\in \Z}$ are positive real numbers chosen such that
		\[ \int_\T  u(x) \mathscr{L}_\gamma (u) d\mu(x) = \sum_{k\in \Z} \gamma_k |\hat u_k|^2 \geq 0. \]
	\end{remark}
	
\section{Numerical results}
	\graphicspath{{picturesNB/}}
	In this part, we illustrate the theorem \ref{thm_global} numerically. 
Our purpose is first to find similar results as in \cite{BonDouKar} i.e. find a $\gamma_k$ constant such that the solution does not blow up. Then build a sequence of $\gamma_k$, still preventing the blow-up, such that $\lim\limits_{|k| \rightarrow +\infty} \gamma_k = 0$. Since dKdV is a low frequencies problem, we do not need to damp all the frequencies.

\subsection{Computation of the damping}

	In order to find the suitable damping, one may use the dichotomy. We remind that our goal is to prevent the blow-up, i.e., avoid that $\lim\limits_{t \rightarrow + \infty} \| u \|_{H^1} = +\infty$. Let us begin finding a constant damping $\mathscr{L}_{\gamma}(u) = \gamma u$ as weak as possible. We mean by weak that $\gamma$ has to be as lower as possible to prevents the blow up. Let $\gamma_a$ respectively $\gamma_e$ be the damping which prevents the explosion and which does not respectively. To initialize the dichotomy, we give a value to $\gamma$ and we determine the initial values of $\gamma_a$ and $\gamma_e$.  Then from these two initial values, we bring them closer by using dichotomy. The method is detailed in algorithm \ref{algo1} and illustrated in Figures \ref{Step1} and \ref{Step2}.

	\begin{algorithm}[H]
		\caption{$\gamma_a$ and $\gamma_e$ using dichotomy}\label{algo1}
		\begin{multicols}{2}
		\begin{algorithmic}[1]
			\REQUIRE $\gamma_0$, $\epsilon$
			\STATE Initialisation of $\gamma$ : $\gamma = \gamma_0$
			\STATE Simulation with $\gamma_k = \gamma$
			\IF{Explosion}
				\WHILE{Explosion}
					\STATE $\gamma = 2\gamma$
					\STATE Simulation with $\gamma_k = \gamma$
				\ENDWHILE
				\STATE $\gamma_a = \gamma$
				\STATE $\gamma_e = \frac{\gamma}{2}$
			\ELSE
				\WHILE{Damping}
					\STATE $\gamma = \frac{\gamma}{2}$
					\STATE Simulation with $\gamma_k = \gamma$
				\ENDWHILE	
				\STATE $\gamma_e = \gamma$
				\STATE $\gamma_a = 2\gamma$
			\ENDIF
			\WHILE{$| \gamma_a - \gamma_e | > \epsilon$}
				\STATE $\gamma = \frac{\gamma_a + \gamma_e}{2}$
				\STATE Simulation with $\gamma_k = \gamma$
				\IF{Explosion}
					\STATE $\gamma_e = \gamma$
				\ELSE
					\STATE $\gamma_a = \gamma$
				\ENDIF
			\ENDWHILE	
		\end{algorithmic}
		\end{multicols}
	\end{algorithm}
	
	\begin{minipage}{0.46\linewidth}
		\shorthandoff{:}\begin{tikzpicture}[scale=0.5]
		\draw[->] (-0.2,0)--(11.5,0) node[right] {$k$};
		\draw[->] (0,-0.2)--(0,9) node[above] {$\gamma_k$};
		\node[below left=0.025cm] at (0,0) {$0\strut$};
		\draw[-,thick] (0,2)--(10,2) node[right] {$\gamma_0$};
		\draw[-,color=red,thick] (0,4)--(10,4) node[right] {$\gamma_e$};
		\draw[-,color=blue,thick] (0,8)--(10,8) node[right] {$\gamma_a$};
		\draw[-,color=blue,thick] (0,1)--(10,1) node[right] {$\gamma_a$};
		\draw[-,color=red,thick] (0,0.5)--(10,0.5) node[right] {$\gamma_e$};
		\draw[->,thick] (2,2)--(2,4);
		\draw[->,thick] (1.75,2)--(1.75,8);
		\node[right] at (2,3) {If explosion with $\gamma_0$};
		\draw[->,thick] (1,2)--(1,1);
		\draw[->,thick] (0.75,2)--(0.75,0.5);
		\node[right] at (1,1.5) {If damping with $\gamma_0$};
		\end{tikzpicture}\shorthandon{:}
		
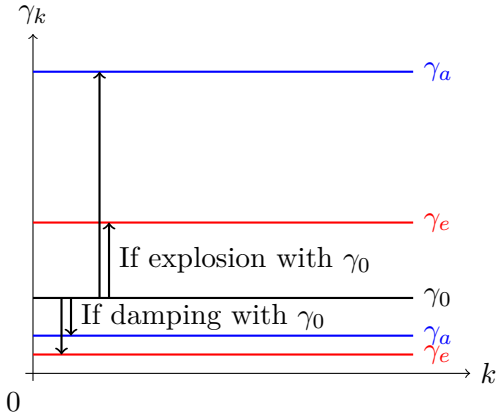
\captionof{figure}{Initialization}
		\label{Step1}
	\end{minipage}
	\hfill
	\begin{minipage}{0.46\linewidth}
		\shorthandoff{:}\begin{tikzpicture}[scale=0.5]
		\draw[->] (-0.2,0)--(10,0) node[right] {$k$};
		\draw[->] (0,-0.2)--(0,9) node[above] {$\gamma_k$};
		\node[below left=0.025cm] at (0,0) {$0\strut$};
		\draw[-,color=blue,thick] (0,8)--(9,8) node[right] {initial $\gamma_a$};
		\draw[-,color=red,thick] (0,2)--(9,2) node[right] {initial $\gamma_e$};
		\draw[dotted,color=blue,thick] (0,5)--(9,5) node[right] {optimal $\gamma_a$};
		\draw[dotted,color=red,thick] (0,4)--(9,4) node[right] {optimal $\gamma_e$};
		\draw[->,color=blue,thick] (4.5,8)--(4.5,5);
		\draw[->,color=red,thick] (4.5,2)--(4.5,4);
		\draw[<->,thick] (6.5,5)--(6.5,4); 
		\node[right] at (6.5,4.5) {$\epsilon$};
		\end{tikzpicture}\shorthandon{:}	
		
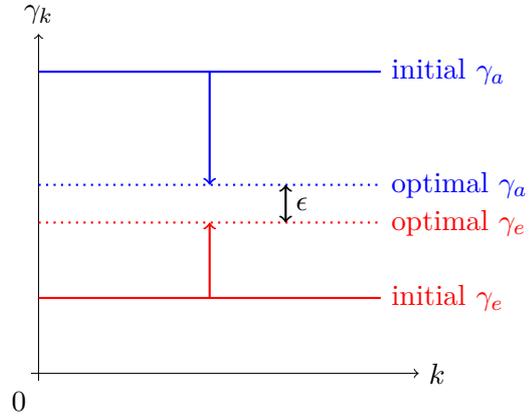
\captionof{figure}{Dichotomy}
		\label{Step2}
	\end{minipage}

	
	W extend the method to frequencies bands in order to build sequencies $\gamma_k$ decreasing with respect to $|k|$ and tending to 0 when $|k|$ tends to the infinity. So we begin by defining the frequencies bands ($N_1 < N_2 < \ldots$) and we proceed as previously but only on the frequencies $|k| \geq N_i$. The method is described in algorithm \ref{algo3} and illustrated in Figure \ref{Step3} and \ref{Step4}.
	
	\begin{algorithm}[H]
		\caption{$\gamma_a$ and $\gamma_e$ on the band using dichotomy}\label{algo3}
		\begin{multicols}{2}
		\begin{algorithmic}[1]
			\REQUIRE $\gamma_a$, $N$ and $Nb\_iter$
			\STATE Initialisation $\gamma = \gamma_a$
			\STATE $\gamma_{|k|>N} = 0$
			\STATE Simulation with $\gamma$
			\IF{Damping}
				\RETURN $\gamma_a = \gamma$
			\ELSE
				\STATE $\gamma = \gamma_a$
				\WHILE{Damping}
					\STATE $\gamma_{|k|>N} = \frac{\gamma_{|k|>N}}{2}$
					\STATE Simulation with $\gamma$
				\ENDWHILE
				\STATE $\gamma_e = \gamma$
				\STATE $\gamma_{a,|k|>N} = 2\gamma_{|k|>N}$
			\ENDIF
			\FOR{$i=1$ to $Nb\_iter$}
				\STATE $\gamma_{|k|>N} = \frac{\gamma_{a,|k|>N} + \gamma_{e,|k|>N}}{2}$
				\STATE Simulation with $\gamma$
				\IF{Explosion}
					\STATE $\gamma_{e,|k|>N} = \gamma_{|k|>N}$
				\ELSE
					\STATE $\gamma_{a,|k|>N} = \gamma_{|k|>N}$	
				\ENDIF
			\ENDFOR
		\end{algorithmic}
		\end{multicols}
	\end{algorithm}
	

	\begin{minipage}{0.46\linewidth}
		\shorthandoff{:}\begin{tikzpicture}[scale=0.5]
		\draw[->] (-0.2,0)--(10,0) node[right] {$k$};
		\draw[->] (0,-0.2)--(0,9) node[above] {$\gamma_k$};
		\node[below left=0.025cm] at (0,0) {$0\strut$};
		\draw[-] (3,-0.1)--(3,0.1);
		\node[below] at (3,0){$N_1$};
		\draw[-,thick,color=blue] (0,8)--(9,8) node[right,text width=2cm] {$\gamma_a$ obtained at step 2};;
		\draw[-,thick,color=blue] (3,6)--(9,6) node[right] {initial $\gamma_a$};
		\draw[-,thick,color=red] (3,2)--(9,2) node[right] {initial $\gamma_e$};
		\draw[dotted] (3,0)--(3,8);
		\draw[dotted,thick,color=blue] (3,5)--(9,5) node[right] {optimal $\gamma_a$};
		\draw[dotted,thick,color=red] (3,4)--(9,4) node[right] {optimal $\gamma_e$};
		\draw[->,thick,color=blue] (6,6)--(6,5);
		\draw[->,thick,color=red] (6,2)--(6,4);
		\end{tikzpicture}\shorthandon{:}
		
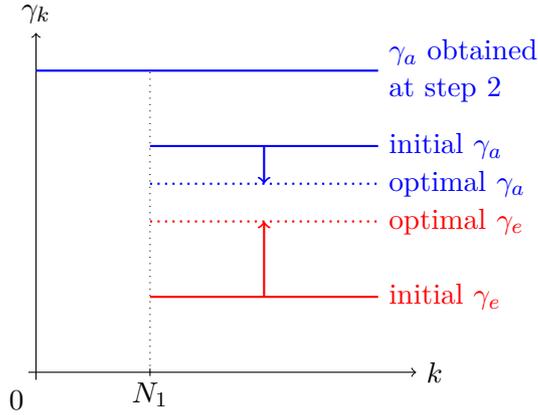
\captionof{figure}{Initialization}
		\label{Step3}
	\end{minipage}
	\hfill
	\begin{minipage}{0.46\linewidth}
		\shorthandoff{:}\begin{tikzpicture}[scale=0.5]
		\draw[->] (-0.2,0)--(10,0) node[right] {$k$};
		\draw[->] (0,-0.2)--(0,9) node[above] {$\gamma_k$};
		\node[below left=0.025cm] at (0,0) {$0\strut$};
		\draw[-] (3,-0.1)--(3,0.1);
		\node[below] at (3,0){$N_1$};
		\draw[dotted] (3,0)--(3,8);
		\draw[-] (4.5,-0.1)--(4.5,0.1);
		\node[below] at (4.5,0){$N_2$};
		\draw[dotted] (4.5,0)--(4.5,6);
		\draw[-,thick,color=blue] (0,8)--(3,8);
		\draw[-,thick,color=blue] (3,6)--(9,6) node[right,text width=2cm] {$\gamma_a$ obtained at step 3};
		\draw[-,thick,color=blue] (4.5,4.5)--(9,4.5) node[right] {initial $\gamma_a$};
		\draw[-,thick,color=red] (4.5,1.75)--(9,1.75) node[right] {initial $\gamma_e$};
		\draw[dotted,thick,color=blue] (4.5,3.5)--(9,3.5) node[right] {optimal $\gamma_a$};
		\draw[dotted,thick,color=red] (4.5,2.75)--(9,2.75) node[right] {optimal $\gamma_e$};
		\draw[->,thick,color=blue] (7.25,4.5)--(7.25,3.5);
		\draw[->,thick,color=red] (7.25,1.75)--(7.25,2.75);
		\end{tikzpicture}\shorthandon{:}
		
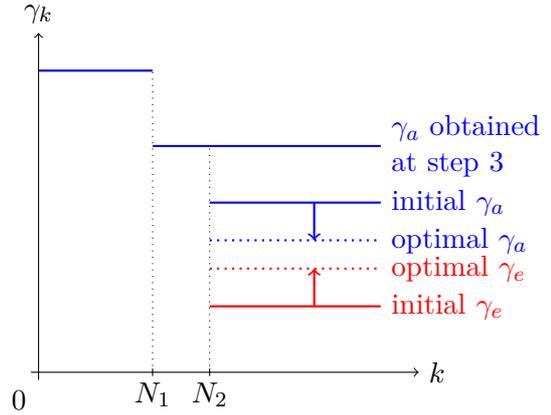
\captionof{figure}{Find the damping}
		\label{Step4}
	\end{minipage}

\subsection{Numerical scheme}	

	Numerous schemes were introduced in [CheSad]. Here we chose a Sanz-Serna scheme for the discretisation in time. In space, we use the FFT. Actually, the scheme is written, for all $k$, as 
	\begin{multline*}
		\left( 1+\frac{\Delta t}{2}(ik-ik^3+\gamma_k) \right) \widehat{u^{(n+1)}}(k) = \left( 1-\frac{\Delta t}{2}(ik-ik^3+\gamma_k) \right) \widehat{u^{(n)}}(k) \\
		- \frac{ik \Delta t}{p+1} \mathscr{F} \left[ \left( \frac{u^{(n+1)}+u^{(n)}}{2} \right)^{p+1} \right](k). 
	\end{multline*}
	We find $\widehat{u^{(n+1)}}_k$ with a fixed-point method. In order to have a good look of the blow-up, we also use an adaptative time step.

\subsection{Simulations}
	
	We consider the domain $[-L,L]$ where $L=50$. We take as initial datum a disturbed soliton, written as 
	\[ u_0(x) = 1.01 \times \left( \frac{(p+1)(p+2)(c-1)}{2} \right)^{\frac{1}{p}} \cosh^{-\frac{2}{p}} \left( \pm \sqrt{\frac{p(c-1)}{4}} (x-ct-d) \right), \]
	where $p=5$, $c=1.5$ and $d=0.2L$. We discretise the space in $2^{11}$ points. The Figure \ref{no_amort} shows the solution whithout damping, i.e., $\gamma_k = 0, \ \forall k$. We observe that the $L^2$-norm of $u_x$ increases strongly and the solution tends to a wavefront (as in \cite{BonDouKar}).
	
	\begin{figure}[H]
		\begin{changemargin}{-1.4cm}{0cm}
		\centering
		\includegraphics[scale=0.5]{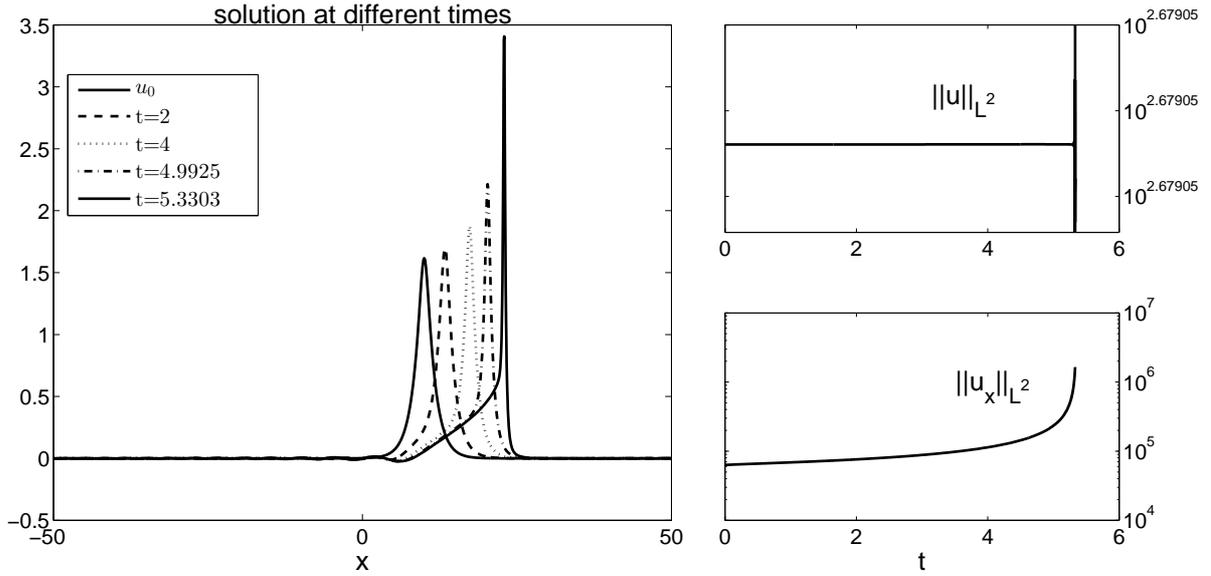} 
		\end{changemargin}	
		\caption{At left, solution at different times $t=$ 0, 2, 4, 4.9925 and 5.3303. At right, $H^1$-norm and $L^2$-norm evolution without damping and a perturbed soliton as initial datum. Here $p=5$.}
		\label{no_amort}
	\end{figure}
	
	Using the methods introduced previously, we first find two optimal constant dampings $\gamma_e = 0.0025$ and $\gamma_a = 0.0027$. As we can see in Figure \ref{amort_explo}, $\gamma_e$ does not prevent the blow up. In the opposite in Figure \ref{amort} $\gamma_a$ does. And we also notice that the two dampings are quite close.
	
	\begin{figure}[H]
		\begin{changemargin}{-1.4cm}{0cm}
		\centering
		\includegraphics[scale=0.5]{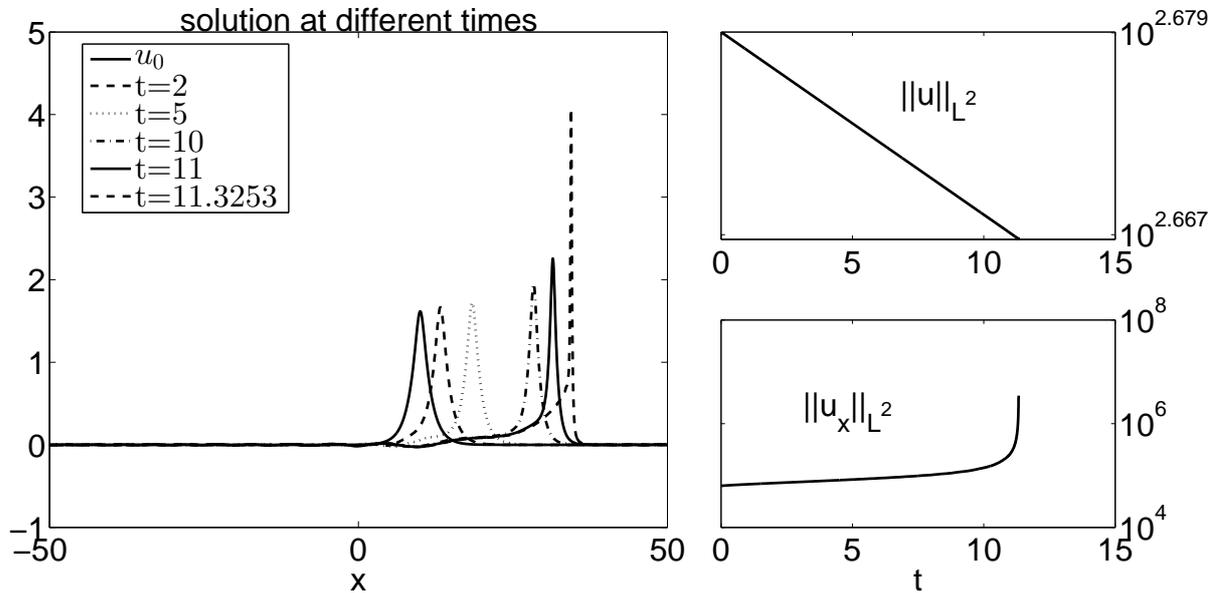} 
		\end{changemargin}	
		\caption{At left, solution at different times $t=$ 0, 2, 5, 10, 11 and 11.3253. At right, $H^1$-norm and $L^2$-norm evolution with $\gamma_k=0.0025$ and a perturbed soliton as initial datum. Here $p=5$.}
		\label{amort_explo}
	\end{figure}

	\begin{figure}[H]
		\begin{changemargin}{-1.4cm}{0cm}
		\centering
		\includegraphics[scale=0.5]{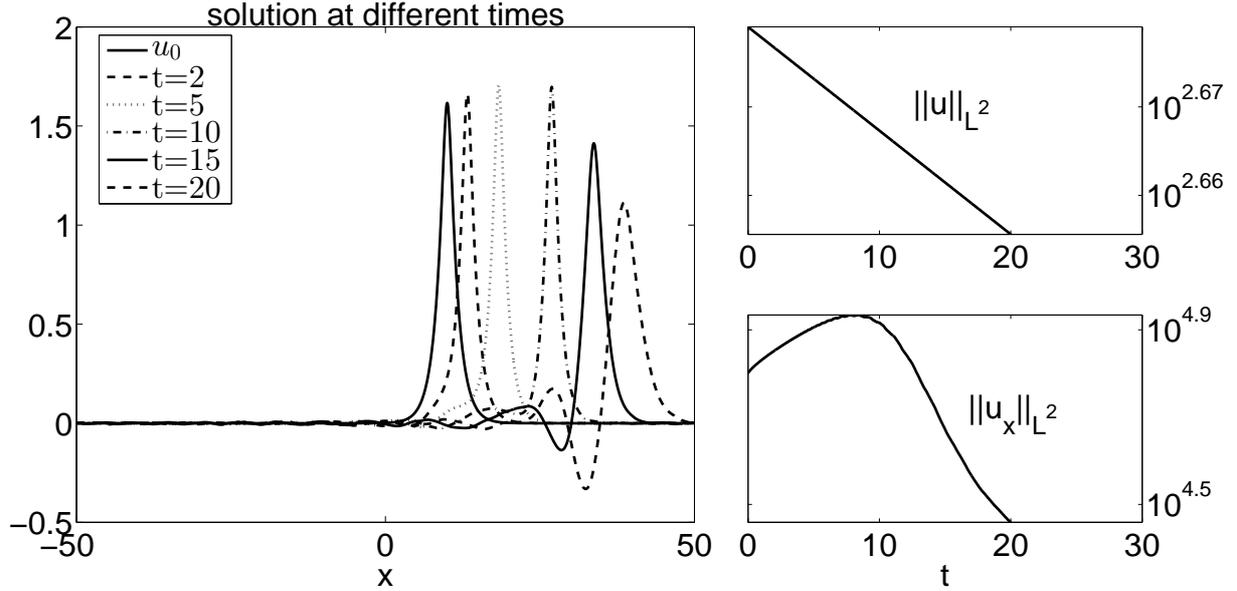} 
		\end{changemargin}	
		\caption{At left, solution at different times $t=$ 0, 2, 5, 10, 15 and 20. At right, $H^1$-norm and $L^2$-norm evolution with $\gamma_k=0.0027$ and a perturbed soliton as initial datum. Here $p=5$.}
		\label{amort}
	\end{figure}
	Considering more general sequences, particularly such that $\lim\limits_{|k|\rightarrow +\infty} \gamma_k = 0$. Using algorithm \ref{algo3}, Figure \ref{soliton_seq} shows that the sequence ($\gamma_a$) as a frontier between the dampings which prevent the blow up and the other which do not. To illustrate this, we take two dampings written as gaussians. The first (denoted by $\gamma_1$) is build to be always above the sequence $\gamma_a$ and the second (denoted by $\gamma_2$) to be always below. In Figures \ref{soliton_seq_amort} and \ref{soliton_seq_explo} we observe the damping $\gamma=\gamma_1$ prevents the blow up. But if we take $\gamma = \gamma_2$, the solution blows-up.
	
	\begin{figure}[H]
		\begin{changemargin}{-1.4cm}{0cm}
		\centering
		\includegraphics[scale=0.5]{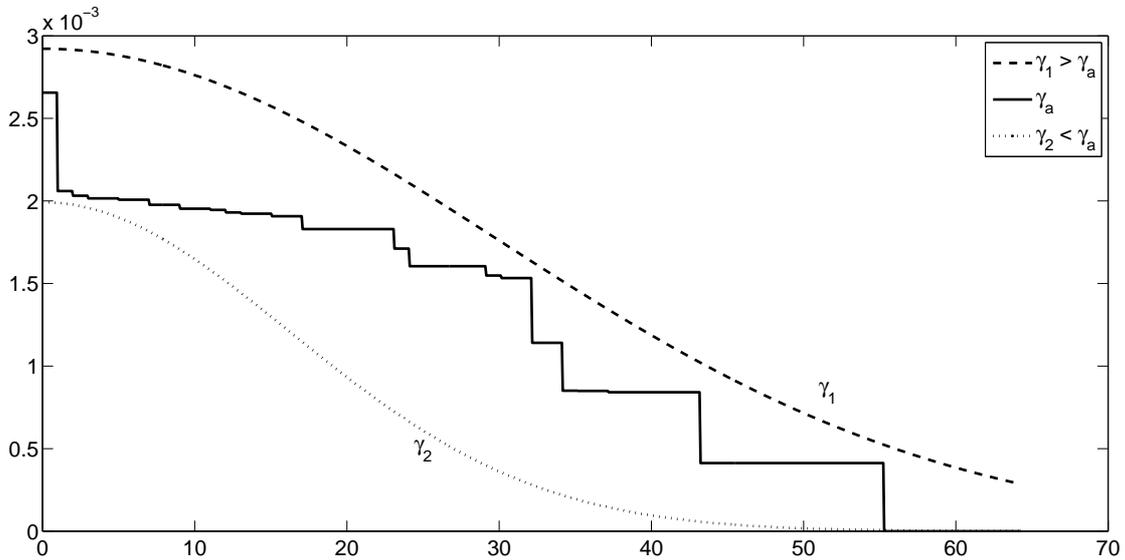} 
		\end{changemargin}	
		\caption{Example of a build damping. Here the initial datum is the perturbed soliton. Here $p=5$.}
		\label{soliton_seq}
	\end{figure}
	
	\begin{figure}[H]
		\begin{changemargin}{-1.4cm}{0cm}
		\centering
		\includegraphics[scale=0.5]{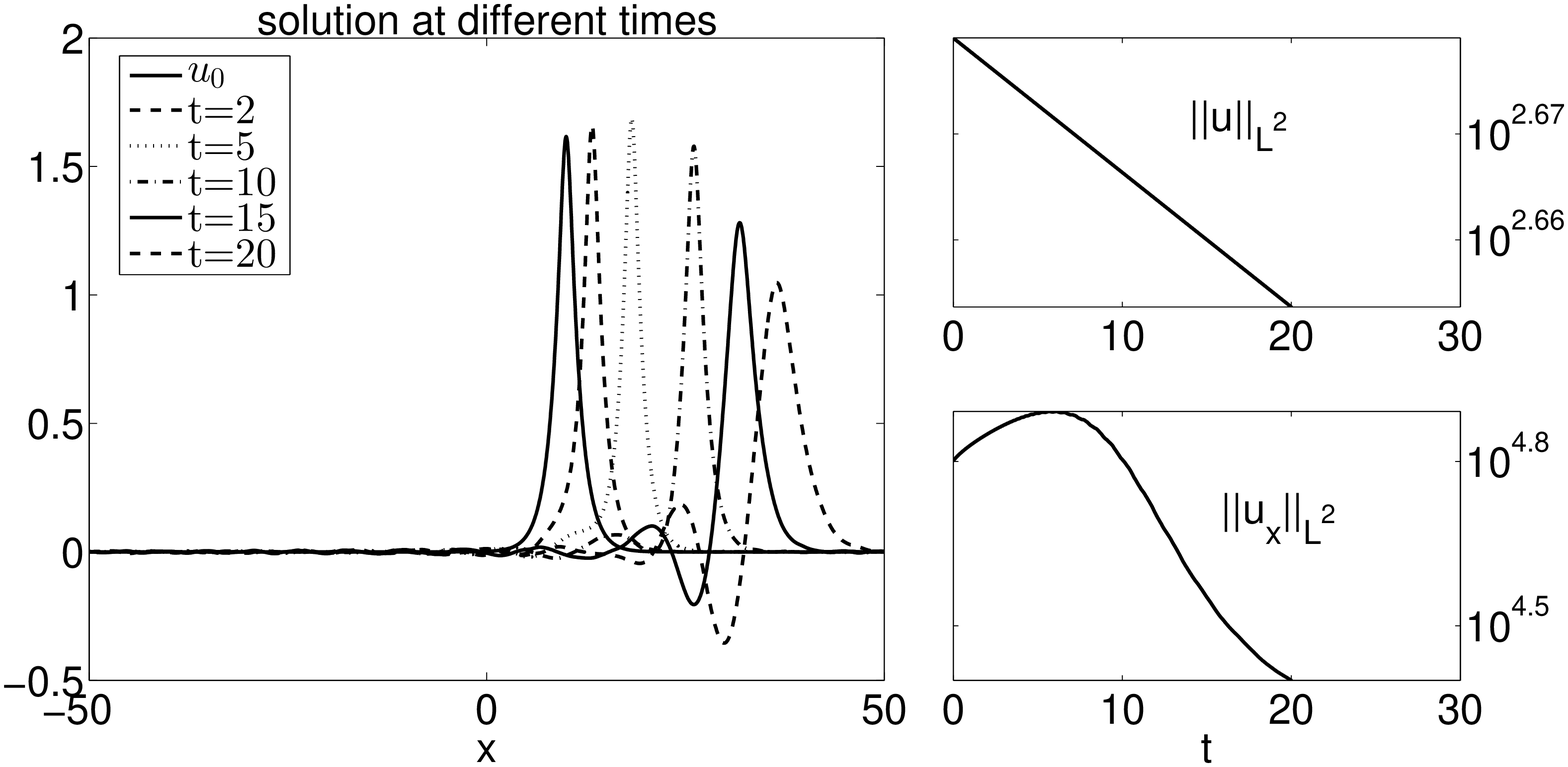} 
		\end{changemargin}	
		\caption{At left, solution at different times $t=$ 0, 2, 5, 10, 15 and 20. At right, $H^1$-norm and $L^2$-norm evolution with $\gamma = \gamma_1$ and a perturbed soliton as initial datum. Here $p=5$.}
		\label{soliton_seq_amort}
	\end{figure}
	
	\begin{figure}[H]
		\begin{changemargin}{-1.4cm}{0cm}
		\centering
		\includegraphics[scale=0.5]{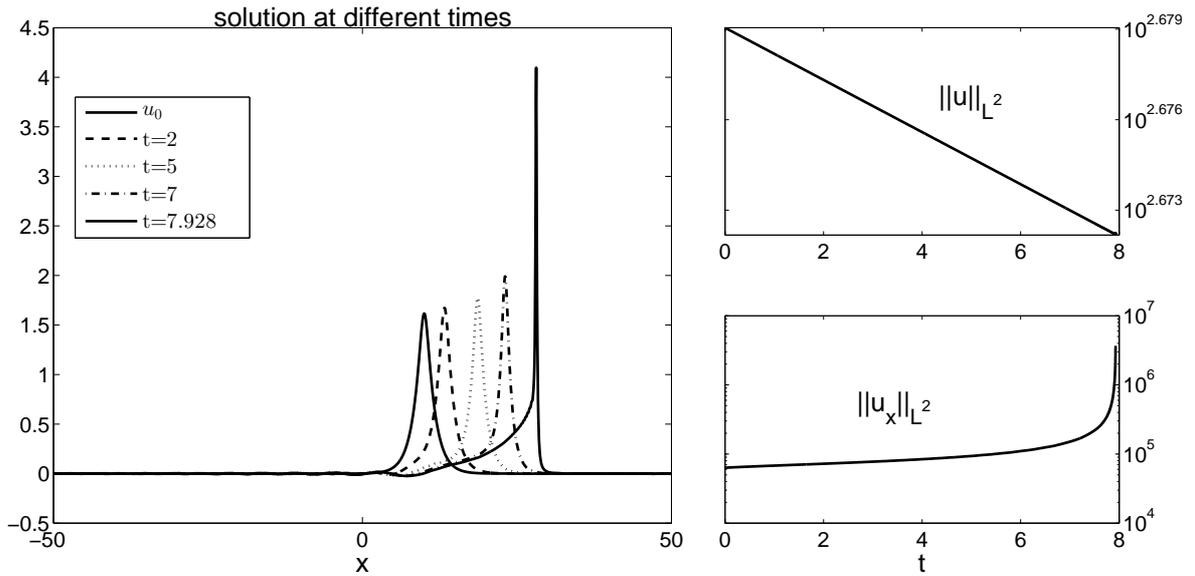} 
		\end{changemargin}	
		\caption{At left, solution at different times $t=$ 0, 2, 5, 7 and 7.928. At right,  $H^1$-norm and $L^2$-norm evolution with $\gamma = \gamma_2$ and a perturbed soliton as initial datum. Here $p=5$.}
		\label{soliton_seq_explo}
	\end{figure}	
	
\section*{Conclusion}
	We studied the behavior of the damped generalized KdV equation. If $p<4$, the solution does not blow-up whereas if $p \geq 4$, it can. To prevent the blow-up, the term $\gamma$ defining the damping has to be large enough. In particular, we build a sequence of $\gamma$ which vanishes for high frequencies.
This frequential approach for the damping seems useful for low frequencies problem.

\vspace{0.5cm}

\noindent\textbf{Acknoledgments.} I would like to thank my thesis supervisors, Jean-Paul Chehab and Youcef Mammeri, for their help and comments.

\bibliographystyle{alpha}
\bibliography{biblio}	

\begin{thebibliography}{BDKM96}

\bibitem[ABS89]{AmBonaSch}
C.J. Amick, J.L. Bona, and M.E. Schonbek.
\newblock Decay of solutions of some nonlinear wave equations.
\newblock {\em J. Differential Equations}, 81(1):1--49, 1989.

\bibitem[BDKM96]{BonDouKar}
J.L. Bona, V.A. Dougalis, O.A. Karakashian, and W.R. McKinney.
\newblock The effect of dissipation on solutions of the generalized
  {K}orteweg-de {V}ries equation.
\newblock {\em J. Comput. Appl. Math.}, 74(1-2):127--154, 1996.
\newblock TICAM Symposium (Austin, TX, 1995).

\bibitem[BS74]{BonaSmith2}
J.~Bona and R.~Smith.
\newblock Existence of solutions to the {K}orteweg-de {V}ries initial value
  problem.
\newblock In {\em Nonlinear wave motion ({P}roc. {AMS}-{SIAM} {S}ummer {S}em.,
  {C}larkson {C}oll. {T}ech., {P}otsdam, {N}.{Y}., 1972)}, pages 179--180.
  Lectures in Appl. Math., Vol. 15. Amer. Math. Soc., Providence, R.I., 1974.

\bibitem[BS75]{BonaSmith1}
J.~L. Bona and R.~Smith.
\newblock The initial-value problem for the {K}orteweg-de {V}ries equation.
\newblock {\em Philos. Trans. Roy. Soc. London Ser. A}, 278(1287):555--601,
  1975.

\bibitem[CR04]{CabRosa}
M.~Cabral and R.~Rosa.
\newblock Chaos for a damped and forced {K}d{V} equation.
\newblock {\em Phys. D}, 192(3-4):265--278, 2004.

\bibitem[CS13a]{CheSad2}
J.-P. Chehab and G.~Sadaka.
\newblock Numerical study of a family of dissipative {K}d{V} equations.
\newblock {\em Commun. Pure Appl. Anal.}, 12(1):519--546, 2013.

\bibitem[CS13b]{CheSad}
J.-P. Chehab and G.~Sadaka.
\newblock On damping rates of dissipative {K}d{V} equations.
\newblock {\em Discrete Contin. Dyn. Syst. Ser. S}, 6(6):1487--1506, 2013.

\bibitem[Ghi88]{Ghi}
J.-M. Ghidaglia.
\newblock Weakly damped forced {K}orteweg-de {V}ries equations behave as a
  finite-dimensional dynamical system in the long time.
\newblock {\em J. Differential Equations}, 74(2):369--390, 1988.

\bibitem[Ghi94]{Ghi2}
J.-M. Ghidaglia.
\newblock A note on the strong convergence towards attractors of damped forced
  {K}d{V} equations.
\newblock {\em J. Differential Equations}, 110(2):356--359, 1994.

\bibitem[Gou00]{Goub}
O.~Goubet.
\newblock Asymptotic smoothing effect for weakly damped forced {K}orteweg-de
  {V}ries equations.
\newblock {\em Discrete Contin. Dynam. Systems}, 6(3):625--644, 2000.

\bibitem[GR02]{GoubRosa}
O.~Goubet and R.M.S. Rosa.
\newblock Asymptotic smoothing and the global attractor of a weakly damped
  {K}d{V} equation on the real line.
\newblock {\em J. Differential Equations}, 185(1):25--53, 2002.

\bibitem[I{\'o}r90]{Iorio}
R.J. I{\'o}rio, Jr.
\newblock Kd{V}, {BO} and friends in weighted {S}obolev spaces.
\newblock In {\em Functional-analytic methods for partial differential
  equations ({T}okyo, 1989)}, volume 1450 of {\em Lecture Notes in Math.},
  pages 104--121. Springer, Berlin, 1990.

\bibitem[KdV95]{KdV}
D.~J. Korteweg and G.~de~Vries.
\newblock Xli. on the change of form of long waves advancing in a rectangular
  canal, and on a new type of long stationary waves.
\newblock {\em Philosophical Magazine Series 5}, 39(240):422--443, 1895.

\bibitem[MM02]{MarMer}
Y.~Martel and F.~Merle.
\newblock Stability of blow-up profile and lower bounds for blow-up rate for
  the critical generalized {K}d{V} equation.
\newblock {\em Ann. of Math. (2)}, 155(1):235--280, 2002.

\bibitem[OS70]{OttSudan}
E.~Ott and R.N. Sudan.
\newblock Damping of solitaries waves.
\newblock {\em Phys. Fluids}, 13(6):1432--1435, 1970.

\end{thebibliography}
\nocite{*}

\end{document}